\documentclass[12pt]{article}
\usepackage{amsmath,amsthm,amssymb}
\usepackage[utf8]{inputenc}
\usepackage{color}
\usepackage[dvipsnames]{xcolor}
\usepackage{multicol}
\usepackage{caption}

\title{Fixed Points of Augmented Generalized Happy Functions} 
\author{
Breeanne Baker Swart \and
Kristen A. Beck \and
Susan Crook \and
Christina Eubanks-Turner \and 
Helen G. Grundman \and
May Mei \and
Laurie Zack 
\thanks{Work on the project was partially supported by the National Science Foundation grant \#DMS 1239280.}}
\date{\today}

\usepackage{mathtools}
\DeclarePairedDelimiter{\floor}{\lfloor}{\rfloor}

\newcommand{\ZZ}{\mathbb{Z}}

\newcommand{\ds}{\displaystyle}

\newcommand{\Scb}{\ensuremath{S_{[c,b]}}}

\theoremstyle{plain}
\newtheorem{theorem}{Theorem}[section]

\newtheorem{lemma}[theorem]{Lemma}
\newtheorem{corollary}[theorem]{Corollary}

\theoremstyle{remark}
\newtheorem{definition}{Definition}

\begin{document}

\maketitle

\begin{abstract}
An augmented generalized happy function $\Scb$ maps a positive integer to the sum of the squares of its base $b$ digits plus $c$. In this paper, we study various properties of the fixed points of $\Scb$; count the number of fixed points of $\Scb$, for $b \geq 2$ and $0<c<3b-3$; and  prove that, for each $b \geq 2$, there exist arbitrarily many consecutive values of $c$ for which $\Scb$ has no fixed point.
\end{abstract}

\section{Introduction}\label{S:Introduction}

The concept of a happy number, defined in~\cite{honsberger} and popularized by~\cite{guy}, was generalized in~\cite{genhappy} by allowing for varying bases and exponents in the defining function.  In~\cite{augment}, this was generalized further, altering the defining function with the addition of a constant.  Specifically,
for integers $c \geq 0$ and $b \geq 2$, the augmented generalized happy function, $\Scb:{\ZZ}^+ \rightarrow {\ZZ}^+$, is defined by
\begin{equation}\label{eq:definition}
\Scb\left(\sum_{i=0}^n a_i b^i \right) =
c + \sum_{i=0}^n a_i^2,
\end{equation}
where $0\leq a_i \leq b-1$ and $a_n \neq 0$. 
Thus, for a positive integer $a$ denoted $a_n\cdots a_1 a_0$ in base $b$, 
\[
\Scb(a_n\ldots a_1 a_0) = c + a_n^2 + \ldots + a_1^2 + a_0^2.
\]

A positive integer $a$ is a happy number if for some $k\in \ZZ^+$, $S_{[0,10]}^k(a) = 1$.  Although $1$ is the sole fixed point of $S_{[0,10]}$, as shown in~\cite{genhappy}, for $b \neq 10$, $S_{[0,b]}$ may have additional fixed points.  Similarly, as shown in~\cite{augment}, when $c > 0$ (and $b \geq 2$), $S_{[c,b]}$ may have multiple fixed points. 

In this work, we study the fixed points of the functions $\Scb$. First, in Section~\ref{S:bFixedPoints}, we prove some preliminary results providing properties of the fixed points and consecutive fixed points of an arbitrary, but fixed, $\Scb$. Then, in Section~\ref{S:CountingFixedPoints}, we discuss the exact number of fixed points of $\Scb$, in terms of $c$ and $b$. Finally, in Section~\ref{S:ArbitraryLengthDesert}, we let $c \geq 0$ vary and prove that for each $b \geq 2$, there are arbitrarily long sequences of consecutive values of $c$ for which $\Scb$ has no fixed point.

\section{Fixed Point Characteristics} \label{S:bFixedPoints}

Here we discuss a variety of results concerning relationships between the fixed points of a single $\Scb$, where $c \geq 0$ and $b \geq 2$ are arbitrary integers. 
The first theorem concerns consecutive fixed points of $\Scb$.

\begin{theorem}\label{t:btwin}
Fix $c \geq 0$ and $b \geq 2$.  
\begin{enumerate}
\item\label{pairiff}
If $a\in \ZZ^+$ is a multiple of $b$, then $a$ is a fixed point of $\Scb$ if and only if $a + 1$ is a fixed point of $\Scb$.
\item\label{pairs}
Every consecutive pair of fixed points of $\Scb$ has a multiple of $b$ as its first member.
\item\label{notriplet}
There is no consecutive triplet of fixed points of $\Scb$.
\end{enumerate}
\end{theorem}

\begin{proof}

We begin with the proof of Part~\ref{pairiff}.  
Since $a$ is a multiple of $b$, we have 
$\Scb(a + 1) = \Scb(a) + 1^2 = \Scb(a) + 1$.  Thus $\Scb(a) = a$ if and only if $\Scb(a + 1) = a + 1$.
 
For Part~\ref{pairs}, assume that $a$ and $a + 1$ are both fixed points of $\Scb$ and 
using standard notation for base $b$, let
\[a = \sum_{i=0}^n a_i b^i.\]   

First, assume that $a_0 \neq b-1$.  Then 
\begin{align*}
a & = \Scb(a) = c + \sum_{i=1}^n a_i^2 + a_0^2,  \quad \textnormal{and so}\\
a + 1 & = \Scb(a + 1) = c + \sum_{i=1}^n a_i^2 +(a_0 + 1)^2 \\
&= a + 2a_0 + 1.
\end{align*}
Thus, $2a_0 = 0$, implying that $a_0 = 0$. Therefore $a$ is a multiple of $b$, as desired.

Next assume, for a contradiction, that $a_0 = b - 1$.  Let $j \in \ZZ^+$ be minimal such that $a_j \neq b - 1$. (If every digit of $a$ is equal to $b-1$, then $a+1=b^k$ for some natural number $k$ and $\Scb(a+1)=1$, so $a+1$ is not a fixed point.)
Then
\begin{equation}\label{e:a}
a = \Scb(a) = c + \sum_{i=j}^n a_i^2 + j(b-1)^2,
\end{equation}
and since
\begin{align}
a + 1 & =  \sum_{i=j+1}^n b^i a_i + (a_j+1)b^j, \quad \textnormal{ we have }
\nonumber \\
a + 1 & = \Scb(a + 1) = c + \sum_{i=j+1}^n a_i^2 + (a_j+1)^2.
\label{e:a+1}
\end{align}
Combining equations~(\ref{e:a}) and~(\ref{e:a+1}) yields
\[a_j^2 + j(b - 1)^2 +1 = (a_j + 1)^2.\]
Thus, $ j(b-1)^2 = 2a_j$.  Since $a_j < b - 1$, $j(b - 1) < 2$ and so 
$j = b - 1 = 1$.  But then $2a_j = 1$, which is a contradiction.

Finally, Part~\ref{notriplet} is immediate from Part~\ref{pairs}.
\end{proof}

Lemma~\ref{l:reflections} provides another pairing of fixed points of $\Scb$.

\begin{lemma}\label{l:reflections}
Fix $c \geq 0$, $b \geq 2$, and $a\in \ZZ^+$ where
\[a = \sum_{i=0}^n a_i b^i,\] in standard base $b$ notation with $a_1 \neq 0$.
Let
\[\tilde{a} = \sum_{i=2}^n {a}_i b^i + (b - a_1)b + a_0.\]  
Then $a$ is a fixed point of $\Scb$ if and only if $\tilde{a}$ is a fixed point of $\Scb$.
\end{lemma}

\begin{proof}
Assume that $a$ and $\tilde{a}$ are as above, and that $a$ is a fixed point of $\Scb$. Then 
\begin{align*} 
\Scb(\tilde{a}) &= \Scb \left(\sum_{i=2}^n a_i b^i + (b-a_1)b + a_0 \right) = c + \sum_{i=2}^n a_i^2 + (b-a_1)^2 + a_0^2 \\
&= c + \sum_{i=0}^n a_i^2 + b^2 - 2a_1 b
= \Scb(a)  + b^2 - 2a_1b 
= a + b^2 - 2a_1b \\ &= \sum_{i = 0}^na_ib^i + (b - 2a_1)b  
= \sum_{i=2}^na_ib^i +(b - a_1)b + a_0
= \tilde{a}.
\end{align*}
Therefore $\tilde{a}$ is also a fixed point of $\Scb$. The converse is immediate by symmetry.
\end{proof}

Finally, we consider the parity of $c$ that is required for $\Scb$ to have a fixed point.

\begin{lemma}\label{L:cbparity}
Fix $c \geq 0$ and $b \geq 2$, and let 
$a = \sum_{i=0}^n a_i b^i$ be a fixed point of $\Scb$, in the usual base $b$ notation.
\begin{enumerate}
\item
If $b$ is odd, then $c$ is even.
\item
If $b$ is even, then $c \equiv \ds \sum_{i=1}^n a_i \pmod 2$.
\end{enumerate}
\end{lemma}

\begin{proof}
For $b$ odd, by~\cite[Lemma 2.3]{augment},   
$\Scb(a) \equiv c + a \pmod 2$, which implies that $c$ is even.
For $b$ even, since $a$ is a fixed point of $\Scb$,
\[a_0 \equiv a = \Scb(a) = c + \sum_{i=0}^n a_i^2 \\
\equiv c + \sum_{i=0}^n a_i \pmod 2.\]
Subtracting $a_0$ from both sides of the congruence yields the result.
\end{proof}

\section{Counting the Number of Fixed Points}\label{S:CountingFixedPoints}

In this section, we consider the {\em number} of fixed points of the function $\Scb$ for fixed $c \geq 0$ and $b \geq 2$. In Corollary~\ref{t:counting}, we provide a formula for the number of fixed points of $\Scb$ for all values of $b$ and a range of values of $c$, depending on $b$.  

We begin by determining the number of fixed points of $\Scb$ of the form $ub^n$, where $0 < u < b$ and $n \geq 0$.  To fix notation, for $c \geq 0$, $b \geq 2$, and $n \geq 0$, let 
\[\mathcal F_{[c,b]}^{(n)} =
\{a = ub^n | 0 < u < b \textnormal{~and~}\Scb(a) = a\}.
\]

In the following three lemmas, we provide conditions under which $\mathcal F_{[c,b]}^{(n)}$ assumes specified values.

\begin{lemma}\label{l:onedigit}
Fix $b\geq 2$.
For $c > 0$, $\mathcal F_{[c,b]}^{(0)}$ is empty, while $\mathcal F_{[0,b]}^{(0)} = \{1\}$.
\end{lemma}

\begin{proof}
Let $0 < a < b$.  Then $a = \Scb(a)$ implies that $c = a - a^2 \leq 0$.  Hence, 
if $c = 0$, we have $a = 1$, and if $c > 0$, we have a contradiction.  
\end{proof}

\begin{lemma}\label{lem:count-ub}
Fix $c \geq 0$ and $b \geq 2$.
The cardinality of $\mathcal F_{[c,b]}^{(1)}$ is
\[ 
\left|\mathcal F_{[c,b]}^{(1)}\right|= \begin{cases}
2 & \text{if $\alpha^2 - \alpha b + c = 0$ for some integer  $1 \leq \alpha < \frac{1}{2}b$,}\\
1 & \text{if  $b^2 = 4c$, and}\\
0 & \text{otherwise.}
\end{cases}
\]
\end{lemma}

\begin{proof}
From the definition of $\mathcal F_{[c,b]}^{(1)}$,
given $a = ub$ with $0 < u < b$, we have
$a\in \mathcal F_{[c,b]}^{(1)}$ if and only if $\Scb(a) = a$ or, equivalently,
$c + u^2 = ub$.  Thus, $\left|\mathcal F_{[c,b]}^{(1)}\right| \neq 0$ if and only if there exists some integer $u$, $0 < u < b$, such that 
\begin{equation}~\label{e:Fcb1}
u^2 - ub + c = 0.
\end{equation}
Since this is a quadratic equation, there are at most two such values of $u$, and at most one if $b^2 = 4c$. 

By Lemma~\ref{l:reflections}, $a = ub$ is a fixed point of $\Scb$ if and only if $\tilde{a} = (b-u)b$ is a fixed point of $\Scb$.  Hence
$a = ub \in \mathcal F_{[c,b]}^{(1)}$ if and only if $(b-u)b \in \mathcal F_{[c,b]}^{(1)}$.
Thus, $\left|\mathcal F_{[c,b]}^{(1)}\right| = 2$ if and only if there are two integer solutions to equation~(\ref{e:Fcb1}) with $0 < u < b$, in which case one of the solutions will satisfy $1 \leq u < \frac{1}{2}b$.  The lemma follows.
\end{proof}

\begin{lemma}\label{lem:count-ubn}
For $c \geq 0$, $b \geq 2$, and $n\geq 2$, the cardinality of $\mathcal F_{[c,b]}^{(n)}$ is
\[ 
\left|\mathcal F_{[c,b]}^{(n)}\right|= \begin{cases}
1 & \text{if $b^{2n} - 4c$ is a nonzero perfect square, and}\\
0 & \text{otherwise.}\\
\end{cases}
\]
\end{lemma}

\begin{proof}
Fix $n\geq 2$, and suppose that $a = ub^n$ is a fixed point of $\Scb$ for some $0<u<b$. Then $ub^n = a = \Scb(a) = c + u^2$, which implies that
\begin{equation}\label{eq:quad}
u = \frac{b^n \pm \sqrt{b^{2n}-4c}}{2}.
\end{equation}
Since $u\in \ZZ^+$, $b^{2n} - 4c$ is a perfect square, and since $u < b$ and $n \geq 2$, 
$b^{2n} - 4c$ is nonzero.  Conversely, if $b^{2n} - 4c$ is a nonzero perfect square, then
$\frac{b^n + \sqrt{b^{2n}-4c}}{2} > b$ and so is not a candidate for $u$, while,
letting $u = \frac{b^n - \sqrt{b^{2n}-4c}}{2}$,
it is easily verified that $a = ub^n \in \mathcal F_{[c,b]}^{(n)}$.
\end{proof}

The following theorem and its proof were inspired by the work of Hargreaves and Siksek \cite{hargreaves} on the number of fixed points of (unaugmented) generalized happy functions. As is standard, we let $r_2(n)$ denote the number of representations of $n\in \ZZ^+$ as the sum of two squares; that is, 
\begin{equation}\label{e:r_2}
r_2(n) =   \left|\{(x,y)\in\mathbb{Z}^2 \mid x^2 + y^2 = n\}\right|.
\end{equation}

\begin{theorem}\label{t:count-2digit}
For $c > 0$ and $b\geq 2$, the number of two-digit fixed points of $\Scb$ is given by
\[
\begin{cases}
\frac{1}{2}r_2(b^2-4c+1) + \left|\mathcal F_{[c,b]}^{(1)}\right| & \text{ if $b$ is odd, and}\\[0.75em]
\frac{1}{4}r_2(b^2-4c+1) + \left|\mathcal F_{[c,b]}^{(1)}\right| & \text{ if $b$ is even.}
\end{cases}
\]
\end{theorem}

\begin{proof}
Note that $a = ub + v$ is a fixed point of $\Scb$, with $0 < u < b$ and $0\leq v < b$ if and only if $ub + v = \Scb(ub + v) = c + u^2 + v^2$.   Define
\[
U = \{(u,v) \in \mathbb{Z}^2 \mid 0 < u,v < b \textnormal{~and~} ub + v = c + u^2 + v^2\}.
\]
By the correspondence $(u,v) \leftrightarrow ub + v$, $|U|$ is equal to the number of two-digit fixed points of $\Scb$ that are not multiples of $b$.  Hence, the number of two-digit fixed points of $\Scb$ is equal to $|U| +|\mathcal F_{[c,b]}^{(1)}|$.

Set 
\[
X = \{(x,y)\in \ZZ^2 \mid y \geq1 \text{ odd, and } x^2 + y^2 = b^2 - 4c + 1  \}.
\]
To see that $|U| = |X|$, consider the functions
$\phi: U \to X$ and $\psi: X \to U$ defined by 
\[
\phi(u,v)=(2u - b,2v - 1) \quad\text{ and }\quad \psi(x,y)=\left(\frac{x+b}{2},\frac{y+1}{2}\right).    
\]
A straightforward calculation, and noting that $2 v - 1 > 0$ and odd, shows that the image of $\phi$ is contained in $X$.  Let $(x,y) \in X$,to see that $\psi(x,y)\in U$, first note that $y$ is odd and $x \equiv b \pmod 2$, and so $\psi(x,y) \in \ZZ^2$.
Next, since  $x^2 < x^2 + y^2 = b^2 - 4c + 1 < b^2$, we have $-b < x < b$, implying that $0 < x + b < 2b$, and thus $0 < (x+b)/2 < b$, as desired. Similarly, $1 \leq y < b$, so $1 \leq (y+1)/2 < b$. Finally, a direct calculation verifies that the needed equation is satisfied.  

Since, as is easily checked, $\phi$ and $\psi$ are inverses, it follows that $|U| = |X|$. 

Now, note that $X$ is a subset of
\[
Z = \{(x,y)\in\mathbb{Z}^2 \mid x^2+y^2=b^2-4c+1\},
\]
and recall that, by equation~(\ref{e:r_2}), $|Z| = r_2(b^2-4c+1)$.

If $b$ is odd and $(x,y) \in Z$, then $b^2-4c+1 \equiv 2 \pmod 4$, and so $y$ must be odd. 
Thus $\varphi_{odd}:Z\rightarrow X$ defined by $(x,y)\mapsto (x,|y|)$ is a 2-to-1 surjective
function.  Hence $|X| = \frac{1}{2} |Z| = \frac{1}{2}r_2(b^2-4c+1)$.

If $b$ is even and $(x,y) \in Z$, then $b^2-4c+1$ is odd, and so exactly one of $x$ and $y$ is odd.  Thus 
$\varphi_{even}:Z\rightarrow X$ defined by 
\[(x,y)\mapsto \left\{
	\begin{array}{ll}
		(x,|y|)  & \mbox{if } y \mbox{ is odd},\\
		(y,|x|)  & \mbox{if } x \mbox{ is odd},
	\end{array}
\right.\]
is a 4-to-1 surjective
function.  Hence $|X| = \frac{1}{4} |Z| = \frac{1}{4}r_2(b^2-4c+1)$.

Recalling that the number of two-digit fixed points of $\Scb$ is $|U| + |\mathcal F_{[c,b]}^{(1)}| = |X| + |\mathcal F_{[c,b]}^{(1)}|$, the result follows. 
\end{proof}

\begin{corollary}\label{t:counting}
For $b\geq 2$ and $0 < c < 3b - 3$, the number of fixed points of $\Scb$ is exactly
\[
\begin{cases}
\frac{1}{2}r_2(b^2-4c+1) + \left|\mathcal F_{[c,b]}^{(1)}\right| & \text{ if $b$ is odd, and}\\[0.75em]
\frac{1}{4}r_2(b^2-4c+1) + \left|\mathcal F_{[c,b]}^{(1)}\right| & \text{ if $b$ is even.}
\end{cases}
\]
\end{corollary}

\begin{proof}
By~\cite[Lemma 2.2]{augment}, since $c < 3b - 3$, for each $a > b^2$, $S(a) < a$ and, therefore, $a$ is not a fixed point.  Hence each fixed point of $\Scb$ has at most two digits.  The corollary now follows directly from
Lemma~\ref{l:onedigit} and Theorem~\ref{t:count-2digit}.
\end{proof}

\section{Fixed Point Deserts} \label{S:ArbitraryLengthDesert}

In this section, we fix the base $b\geq 2$ and consider consecutive values of $c$ for which $\Scb$ has no fixed points. 
Note that for a fixed $b$, if
$a = \sum_{i=0}^{n}a_ib^i$ is a fixed point of $\Scb$, with $0\leq a_i < b$, for each $i$, then, solving for $c$, we have
\begin{equation}
 c = \sum_{i=0}^n a_i(b^i - a_i). \label{eqn:cfixedpoint} 
\end{equation}

\begin{definition} For $b\geq 2$ and $k \in \ZZ^+$, an \emph{$k$-desert base $b$} is a set of $k$ consecutive non-negative integers $c$ for each of which $\Scb$ has no fixed points. A \emph{desert base $b$} is an $k$-desert base $b$ for some $k \geq 1$.  \end{definition}

For example, for $28 \leq c \leq 35$, $S_{[c,10]}$ has no fixed points and, therefore, there is an 8-desert base 10 starting at $c=28$.  

We begin by determining bounds on the values of $c$ such that $\Scb$ has a fixed point of a given number of digits.  
For $b\geq 2$ and $n\geq 2$ define
\begin{align*}
m_{b,n} & = b^n - b^2 + 3b  - 3, \text{~and~} \\
M_{b,n} & = b^{n+1}-b^2 - (n-1)(b-1)^2 + (b-\floor*{b/2})\floor*{b/2}.
\end{align*}

\begin{theorem}\label{t:maxmin}
Let $b\geq 2$ and $n \geq 2$.  If $\Scb$ has a $n + 1$-digit fixed point, then $ m_{b,n} \leq c \leq M_{b,n}$.  Further, these bounds are sharp.
\end{theorem}

\begin{proof}
Let $b\geq 2$ and $n \geq 2$ be fixed. By equation~(\ref{eqn:cfixedpoint}), each fixed point $a$ of $\Scb$ determines the value
of $c$.  Treating the $a_i$ in equation~(\ref{eqn:cfixedpoint}) as independent
variables taking on integer values between 0 and $b-1$, inclusive, we find the minimal possible value of $c$ by minimizing each term.
Observe that $a_0(b^0 - a_0)$ is minimal when $a_0 = b - 1$;
for $0 < i < n$, $a_i(b^i - a_i)$ is minimal when $a_i = 0$; and, since $a_n \neq 0$, 
$a_n(b^n-a_n)$ is minimal when $a_n = 1$.  Hence the minimal value of $c$ is determined by 
\[a = \sum_{i=0}^{n}a_i^\prime b^i, \text{ where } 
a_i^\prime = \begin{cases} 1, & \text{for }\ i=n; \\
0, & \text{for }\ 1 \leq i \leq n-1; \\
b-1, & \text{for }\ i=0,
\end{cases}\]
and so, the minimal value of $c$ is
\[c = (b-1)(b^0 - (b-1)) + 1 \cdot (b^n-1) = b^n - b^2 + 3b  - 3 = m_{b,n}.\]

Similarly, maximizing the terms of equation~(\ref{eqn:cfixedpoint}), we find that 
$a_0(b^0 - a_0)$ is maximal when $a_0 = 0$;
$a_1(b^1 - a_1)$ is maximal when $a_1 = \floor*{b/2}$; and
for $1 < i \leq n$, $a_i(b^i - a_i)$ is maximal when $a_i = b - 1$. 
Hence, the maximal value of $c$ is determined by 
\[a = \sum_{i=0}^{n}a_i^{\prime\prime}b^i, \text{ where } 
a_i^{\prime\prime} = \begin{cases} b-1, & \text{for }\ 2 \leq i \leq n; \\
\floor*{\frac{b}{2}}, & \text{for }\ i=1; \\
0, & \text{for }\ i=0,
\end{cases}\]
and, therefore, the maximal value of $c$ is
\begin{align*}
c &= \floor*{b/2}(b^1 - \floor*{b/2}) +
\sum_{i=2}^n (b^i-(b-1))(b-1) \\
&= b^{n+1}-b^2 - (n-1)(b-1)^2 + (b-\floor*{b/2})\floor*{b/2} = M_{b,n}. \qedhere
\end{align*}
\end{proof}

The following lemma is used to prove Theorem~\ref{thm:desert}, which states that for each $b \geq 2$ there exist arbitrarily long deserts base $b$. 

\begin{lemma}\label{lem:desert}
Let $b\geq 2$ and $n \geq 2$. Then, between the numbers $M_{b,n}$ and $m_{b,n+1}$, there exists an $k$-desert base $b$, where 
\[k = m_{b,n+1} - M_{b,n} - 1 > (n-5/4)(b-1)^2.\]
\end{lemma}

\begin{proof}
Let $b \geq 2$ and $n\geq 2$ be fixed.  Note that 
\begin{align*}
m_{b,n+1} - M_{b,n} - 1 &=(b^{n+1} - b^2 + 3b  - 3)  - \\
&\hphantom{mmm}\left(b^{n+1}-b^2 - (n-1)(b-1)^2 + (b-\floor*{b/2})\floor*{b/2}\right) - 1\\
& \geq 3b-3 + (n-1)(b-1)^2 - b^2/4 -1 \\
& = (n-1)(b-1)^2 - b^2/4+b/2-1/4 + 5b/2 - 15/4\\
& > (n-1)(b-1)^2 - (b-1)^2/4\\
& = (n-5/4)(b-1)^2,
\end{align*}
since $b \geq 2$.
Thus, 
\begin{equation}\label{m>M}
m_{b,n+1} >  M_{b,n} + 1.
\end{equation}
Recall that $M_{b,n}$ is an upper bound on values of $c$ such that $\Scb$ has a $(n + 1)$-digit fixed point.  Since $M_{b,x}$ increases as $x$ increases, $M_{b,n}$ is an upper bound on values of $c$ such that $\Scb$ has a fixed point with less than or equal to $(n + 1)$-digits.  Similarly, $m_{b,n+1}$ is a lower bound on values of $c$ such that $\Scb$ has a fixed point with greater than or equal to $(n + 2)$-digits.  

Thus, by equation~(\ref{m>M}), there is no value of $c$  between $M_{b,n}$ and $m_{b,n+1}$ such that $\Scb$ has a fixed point of any size.  Hence there exists an $k$-desert between these two numbers, where $k = m_{b,n+1} - M_{b,n} - 1$.
\end{proof}

\begin{theorem} \label{thm:desert} For each $b \geq 2$ and
$k \in \ZZ^+$, there exists an $k$-desert base $b$. 
\end{theorem} 

\begin{proof}
Fix $b \geq 2$ and $k \in \ZZ^+$.  Since $(n-5/4)(b-1)^2$ is an increasing linear function of $n$, there exists some $n \geq 2$ such that $(n-5/4)(b-1)^2 \geq k$.  It follows from 
Lemma~\ref{lem:desert} that there exists an $k$-desert base $b$.
\end{proof}

\end{document}